\documentclass[a4paper,11pt]{smfart}

\usepackage[OT2,T1]{fontenc}
\usepackage[english,francais]{babel}
\usepackage{amssymb}
\usepackage{amsmath}
\usepackage{amscd}
\usepackage[all]{xy}

\usepackage{verbatim}

\DeclareSymbolFont{rsfs}{U}{rsfs}{m}{n}
\DeclareSymbolFontAlphabet{\mathcal}{rsfs}

\DeclareFontEncoding{OT2}{}{} 
\DeclareTextFontCommand{\textcyr}{\fontencoding{OT2}
    \fontfamily{wncyr}\fontseries{m}\fontshape{n}\selectfont}
\newcommand{\Sha}{\textcyr{Sh}}

\theoremstyle{plain}

\newtheorem{theo}{Th\'eor\`eme}[section]

\newtheorem{prop}[theo]{Proposition}

\newtheorem{cor}[theo]{Corollaire}
\newtheorem{lem}[theo]{Lemme}
\newenvironment{ex}{\noindent{\textbf{Exemple :}}}{}

\def\proofname{D\'emonstration}

\theoremstyle{definition}

\newtheorem{rem}[theo]{Remarque}

\def \Romannumeral #1 {\expandafter\uppercase\expandafter {\romannumeral #1} }

\def \Br {{\rm{Br}}}

\def \A{{\mathbf A}}

\def \Spec {{\rm{Spec\,}}}

\def \SL {{\rm {SL}}}

\def \PGL {{\rm {PGL}}}

\def\Q{{\mathbf{Q}}}
\def\Z{{\mathbf{Z}}}

\DeclareMathOperator{\coker}{coker}

\def\sc{^{\mathrm{sc}}}

\newcommand{\ensemble}[1]{\ensuremath{\mathbf #1} \xspace}

\def\L{\ensemble{L}}

\def\A{{\mathbf{A}}}

\newcommand{\Ext}{\textup{Ext}}

\def\Br{{\rm Br}}
\def\Pic{{\rm Pic}}

\def\G{{\mathbf{G}}}

\title[Points entiers sur les espaces homog\`enes \`a stabilisateurs finis]{        
Obstructions de Brauer-Manin enti\`eres sur les espaces homog\`enes \`a stabilisateurs finis nilpotents}

\author{Cyril Demarche}
\address{Sorbonne Universit\'es, UPMC Univ Paris 06, UMR 7586, Institut de Math\'ematiques de Jussieu-Paris Rive Gauche, Case 247, 4 place Jussieu, F-75005, Paris, France}
\address{Univ Paris Diderot, Sorbonne Paris Cit\'e, UMR 7586, IMJ-PRG, F-75013 Paris, France}
\address{CNRS, UMR7586, IMJ-PRG, F-75005 Paris, France}
\email{demarche@math.jussieu.fr}
\thanks{Ce travail a b\'en\'efici\'e d'une aide de l'Agence Nationale de la Recherche portant la r\'ef\'erence ANR-12-BL01-0005.}
\subjclass{Primary: 11E72; Secondary: 14G05, 14M17, 20G30}

\begin{document}

\date{\today}

\maketitle

\begin{abstract}
Soit $k$ un corps de nombres. On construit des espaces homog\`enes de $\SL_{n,k}$ \`a stabilisateurs finis nilpotents non commutatifs pour lesquels l'obstruction de Brauer-Manin est insuffisante pour expliquer le d\'efaut d'approximation forte (resp. le d\'efaut du principe de Hasse entier).
\end{abstract}

\selectlanguage{english}

\begin{abstract}
Let $k$ be a number field. We construct homogeneous spaces of $\SL_{n,k}$ with finite nilpotent non-abelian stabilizers for which the Brauer-Manin obstruction does not explain the failure of strong approximation (resp. the failure of the integral Hasse principle).
\end{abstract}

\selectlanguage{french}

\def\pn{\par\noindent}

\bigskip

\setcounter{section}{-1}

\section{Introduction}

Si $k$ est un corps de nombres et $H$ est un $k$-groupe alg\'ebrique fini, il serait tr\`es int\'eressant de conna\^itre pour quels ensembles finis $S$ de places de $k$ l'application naturelle
$$H^1(k,H) \to \prod_{v \in S} H^1(k_v,H)$$
est surjective.
Cette question (parfois appel\'ee probl\`eme de Grunwald), motiv\'ee par exemple par le probl\`eme inverse de Galois et par celui de l'existence de corps de nombres \`a ramification prescrite, est intimement li\'ee \`a l'\'etude de l'approximation faible sur l'espace homog\`ene $\SL_{n,k}/H$ d\'efini par une repr\'esentation fid\`ele $H \to \SL_{n,k}$ (voir par exemple \cite{H1}), et plus pr\'ecis\'ement \`a l'\'etude de l'obstruction de Brauer-Manin \`a ladite approximation faible sur cette vari\'et\'e. Ce probl\`eme est encore loin d'\^etre r\'esolu (voir par exemple \cite{H1} ou \cite{D}), tout comme son analogue concernant le principe de Hasse sur des espaces homog\`enes \`a stabilisateurs g\'eom\'etriques finis.

Dans cette note, on s'int\'eresse aux versions enti\`eres de ces probl\`emes, \`a savoir l'approximation forte et le principe de Hasse entier sur les espaces homog\`enes du groupe $\SL_{n,k}$ \`a stabilisateurs finis sur un corps de nombres $k$. 

On montre que pour de telles vari\'et\'es alg\'ebriques, si les stabilisateurs sont des groupes constants nilpotents non commutatifs et si le corps de base contient suffisamment de racines de l'unit\'e, l'obstruction de Brauer-Manin (m\^eme en tenant compte de la partie dite "transcendante" du groupe de Brauer) ne permet pas d'expliquer le d\'efaut d'approximation forte (th\'eor\`eme \ref{theo AF}) ou du principe de Hasse entier (th\'eor\`eme \ref{theo PH}). Autrement dit, concernant l'approximation forte, pour tout $p$-groupe fini non commutatif $H$ et tout morphisme injectif de groupes $H \to \SL_{n,k}$, l'espace homog\`ene quotient $X := \SL_{n,k}/H$ poss\`ede des points ad\'eliques v\'erifiant les conditions de Brauer-Manin qui ne peuvent \^etre approxim\'es pour la topologie ad\'elique par des points rationnels de $X$ (hors d'un ensemble fini fix\'e de places de $k$); concernant le principe de Hasse entier, sous les m\^emes hypoth\`eses, il existe un mod\`ele entier $\mathcal{X}$ de $X$, lisse s\'epar\'e de type fini sur un anneau d'entiers de $k$, tel que $\mathcal{X}$ poss\`ede des points entiers localement partout v\'erifiant les conditions de Brauer-Manin, mais tel que $\mathcal{X}$ n'admette cependant pas de point entier global. 

En particulier, cela montre combien l'arithm\'etique des espaces homog\`enes \`a stabilisateurs finis diff\`ere de celle des espaces homog\`enes \`a stabilisateurs connexes ou ab\'eliens, o\`u au contraire l'obstruction de Brauer-Manin est la seule obstruction \`a l'approximation forte et au principe de Hasse entier (voir par exemple \cite{CTX} ou \cite{BD}). 

Ces r\'eponses n\'egatives \`a la suffisance des obstructions de Brauer-Manin enti\`eres ne permettent toutefois pas d'aborder les questions analogues concernant le principe de Hasse rationnel et l'approximation faible, qui demeurent donc toujours largement ouvertes.

\subsection*{Remerciements} Je remercie vivement Jean-Louis Colliot-Th\'el\`ene, Yonatan Harpaz et Giancarlo Lucchini Arteche pour leurs pr\'ecieux commentaires et leur int\'er\^et pour ce texte.

\subsection*{Quelques notations}
Dans tout ce texte, $k$ est un corps de caract\'eristique nulle, $\overline{k}$ une cl\^oture alg\'ebrique fix\'ee de $k$, $\Gamma_k$ d\'esigne le groupe de Galois de l'extension $\overline{k}/k$. Dans toute cette note, la cohomologie utilis\'ee est la cohomologie \'etale.

Si $k$ est un corps de nombres, on note $\Omega_k$ l'ensemble des places de $k$; si $v \in \Omega_k$, on note $k_v$ le compl\'e\'e de $k$ en $v$ et $\mathcal{O}_v$ son anneau des entiers (par convention, $\mathcal{O}_v := k_v$ si $v$ est une place non archim\'edienne); si $S$ est un ensemble fini de places de $k$, on note $\mathcal{O}_{k,S}$ l'anneau des $S$-entiers de $k$ et $\A_k^S$ l'anneau des ad\`eles hors de $S$. 
Si $X$ est une $k$-vari\'et\'e, on note $\Br(X) := H^2(X,\G_m)$ le groupe de Brauer de $X$. On renvoie \`a \cite{Sko} pour la d\'efinition de l'accouplement de Brauer-Manin $X(\A_k) \times \Br(X) \to \Q / \Z$ et du sous-ensemble $X(\A_k)^\Br \subset X(\A_k)$, ainsi qu'\`a la premi\`ere section de \cite{CTX} pour les g\'en\'eralit\'es sur l'approximation forte et l'obstruction de Brauer-Manin enti\`ere.

\section{Groupe de Brauer}
Dans cette section, on montre le r\'esultat simple suivant, qui g\'en\'eralise les r\'esultats classiques sur la partie alg\'ebrique du groupe de Brauer dans un cas particulier. Si $k$ est un corps et $H$ un $k$-groupe alg\'ebrique, on note $\Ext^c_k(H, \G_m)$ le groupe des extensions centrales de $k$-groupes alg\'ebriques de $H$ par $\G_m$.

\begin{prop} \label{prop Brauer}
Soit $k$ un corps de caract\'eristique nulle et $G$ un $k$-groupe alg\'ebrique semi-simple simplement connexe. Soit $H$ un $k$-sous-groupe constant fini de $G$, $X := G/H$ l'espace homog\`ene correspondant et $\pi : G \to X$ la projection.
On d\'efinit $\Br(X,G) := \ker(\Br(X) \xrightarrow{\pi^*} \Br(G))$.

Alors on dispose d'un isomorphisme canonique de groupes ab\'eliens
$$\Delta'_X : \Ext^c_k(H, \G_m) \xrightarrow{\cong} \Br(X,G) \cong \Br(X)/\Br(k) \, ,$$
qui est fonctoriel en $H$ au sens suivant : si $K \subset H$ est un sous-groupe de $H$, alors le morphisme canonique $f : Y := G/K \to X = G/H$ induit un diagramme naturel
\begin{displaymath}
\xymatrix{
\Ext^c_k(H, \G_m)\ar[r]^{\Delta'_X} \ar[d] & \Br(X)/\Br(k) \ar[d]^{f^*} \\
\Ext^c_k(K, \G_m)\ar[r]^{\Delta'_Y} & \Br(Y)/\Br(k)
}
\end{displaymath}
qui est commutatif.
\end{prop}

\begin{rem}
On dispose \'egalement du morphisme naturel $\Delta_X : \Ext^c_k(H, \G_m) \to \Br(X,G) \cong \Br(X)/\Br(k)$ qui est donn\'e par le cobord en cohomologie \'etale (voir par exemple \cite{CTX}, p.313--314). Il est probable que les morphismes $\Delta_X$ et $\Delta'_X$ co\"incident (au signe pr\`es), mais une telle comparaison n'est pas n\'ecessaire dans ce texte.
\end{rem} 

\begin{proof}
On consid\`ere la suite spectrale de \v{C}ech relative au recouvrement (\'etale) donn\'e par le morphisme $\pi : G \to X$ (qui est un $X$-torseur sous le groupe $H$) :
$$E_2^{p,q} = \check{H}^p(G/X, \underline{H}^q(\G_m)) \implies H^{p+q}(X, \G_m) \, .$$

On obtient une suite exacte en bas degr\'e de la forme
$$\check{H}^0(G/X, \underline{H}^1(\G_m)) \to \check{H}^2(G/X, \G_m) \to \ker(\Br(X) \xrightarrow{\pi^*} \Br(G)) \to \check{H}^1(G/X, \underline{H}^1(\G_m)) \, .$$

On montre alors que $\check{H}^i(G/X, \underline{H}^1(\G_m)) = 0$ pour tout $i \geq 0$, en utilisant le fait que 
$$\Pic(G \times_X G \times_X \dots \times_X G) = \Pic(G \times_k H \times_k \dots \times_k H) = 0 \, ,$$
puisque $\Pic(G) = 0$ et $H$ est fini constant. Par cons\'equent, la suite exacte pr\'ec\'edente induit un isomorphisme 
$$ \check{H}^2(G/X, \G_m) \xrightarrow{\cong} \ker(\Br(X) \xrightarrow{\pi^*} \Br(G)) = \Br(X,G) \, .$$

Identifions maintenant les groupes $\check{H}^2(G/X, \G_m)$ et $\Ext^c_k(H, \G_m)$ : pour cela, en utilisant de nouveau les isomorphismes canoniques $G \times_k H \times_k \dots \times_k H \xrightarrow{\cong} G \times_X G \times_X \dots \times_X G$, on voit que $\check{H}^2(G/X, \G_m)$ s'identifie naturellement au groupe de cohomologie (au sens de la cohomologie des groupes usuelle) $H^2(H, k[G]^*) = H^2(H,k^*)$, o\`u $H$ agit trivialement sur $k^*$. 

Ensuite, puisque $H$ est un groupe constant, on a un isomorphisme de groupes ab\'eliens $H^2(H,k^*) \cong H^2_0(H, \G_m)$, o\`u le second groupe est le groupe de cohomologie de Hochschild pour l'action triviale du $k$-groupe $H$ sur le $k$-groupe $\G_m$. 

Enfin, un r\'esultat classique (voir par exemple \cite{DG}, II.3, proposition 2.3) assure que ce dernier groupe de cohomologie de Hochschild s'identifie canoniquement au groupe $\Ext^c_k(H, \G_m)$. 

On a donc construit un isomorphisme canonique $\Delta'_X : \Ext^c_k(H, \G_m) \xrightarrow{\cong} \Br(X,G)$, dont il est facile de v\'erifier la fonctorialit\'e en $H$.

Enfin, l'isomorphisme $\Br(X,G) \cong \Br(X)/\Br(k)$ est une cons\'equence du th\'eor\`eme principal de \cite{G} et de son corollaire, qui assurent que $\Br(k) \cong \Br(G)$.
\end{proof}

\section{Obstruction de Brauer-Manin \`a l'approximation forte}

Soit $k$ un corps de nombres. On consid\`ere l'espace homog\`ene $X := G/H$, o\`u $G$ est un $k$-groupe alg\'ebrique semi-simple simplement connexe, et $H$ un $k$-groupe fini nilpotent.
Si $S$ est un ensemble fini de places de $k$, on note $\pi^S : X(\A_k) \to X(\A_k^S)$ la projection des points ad\'eliques de $X$ vers les points ad\'eliques hors de $S$ dans $X$, et $\overline{X(k)}^S$ l'adh\'erence de $X(k)$ dans $X(\A^S)$ (muni de la topologie ad\'elique).

\begin{theo} \label{theo AF}
Soit $p$ un nombre premier. Pour tout groupe fini \emph{non commutatif} $H$ d'ordre $p^n$, pour tout ensemble fini $S_0$ de places de $k$, si $k$ contient les racines $p^{n+1}$-i\`emes de l'unit\'e, alors il existe un point ad\'elique $x \in X(\A_k)^\Br$ tel que $\pi^{S_0}(x) \notin \overline{X(k)}^{S_0}$.

Autrement dit, "l'obstruction de Brauer-Manin \`a l'approximation forte" sur $X$ n'est pas la seule.
\end{theo}

\begin{rem}
On ne peut pas v\'eritablement parler d'obstruction de Brauer-Manin \`a l'approximation forte en g\'en\'eral (notamment quand le groupe de Brauer de $X$ modulo le groupe de Brauer de $k$ est infini), c'est pourquoi l'expression est entre guillemets dans l'\'enonc\'e. Il faut donc comprendre l'\'enonc\'e "l'obstruction de Brauer-Manin \`a l'approximation forte sur $X$ n'est pas la seule" au sens de la d\'efinition 2.4 de \cite{CTX2}, c'est-\`a-dire au sens qui est explicit\'e dans l'\'enonc\'e du th\'eor\`eme.
\end{rem}

\begin{proof}
La proposition \ref{prop Brauer} assure que l'on a un isomorphisme canonique
$$\Delta'_X : \textup{Ext}^c_k(H, \G_m) \to \Br(X) / \Br(k) \, ,$$
qui est fonctoriel en $H$ au sens pr\'ecis\'e dans l'\'enonc\'e de la proposition \ref{prop Brauer}.
Puisque $H$ est nilpotent, il existe un sous-groupe central cyclique non trivial $Z$ contenu dans le sous-groupe d\'eriv\'e $D(H)$. On consid\`ere alors la suite exacte naturelle
$$1 \to Z \to H \to H' := H/Z \to 1 \, .$$

Cette suite induit le complexe suivant
\begin{equation} \label{sec ext}
\Ext^c_k(H', \G_m) \to \Ext^c_k(H, \G_m) \to \Ext^c_k(Z, \G_m) \, .
\end{equation}

Montrons le lemme cl\'e suivant :
\begin{lem} \label{lem morph brauer nul}
Sous les hypoth\`eses du th\'eor\`eme \ref{theo AF}, le morphisme naturel $\Ext^c_k(H, \G_m) \to \Ext^c_k(Z, \G_m)$ est le morphisme nul.
\end{lem}

\begin{proof}
On constate d'abord que le groupe $\Ext^c_k(H,\G_m) \cong H^2(H, k^*)$ est un groupe ab\'elien de $p^n$-torsion, car $H$ est de cardinal $p^n$. En outre, on v\'erifie que la suite naturelle de morphismes de groupes
$$\Ext^c_k(H,\mu_{p^n}) \to \Ext^c_k(H,\G_m) \xrightarrow{[p^n]} \Ext^c_k(H,\G_m)$$
est exacte. Cela assure donc que le morphisme 
$$\Ext^c_k(H,\mu_{p^n}) \to \Ext^c_k(H,\G_m)$$
est surjectif.

Soit une extension centrale
$$1 \to \G_m \to H_1 \to H \to 1 \, .$$
On sait donc qu'il existe une extension centrale 
$$1 \to \mu_{p^n} \to H_2 \to H \to 1$$
et un diagramme commutatif de suites exactes courtes
\begin{displaymath}
\xymatrix{
1 \ar[r] & \mu_{p^n} \ar[d]^i \ar[r] & H_2 \ar[d] \ar[r] & H \ar[d]^= \ar[r] & 1 \\
1 \ar[r] & \G_m \ar[r] & H_1 \ar[r] & H \ar[r] & 1 \, ,
}
\end{displaymath}
o\`u $i$ est l'inclusion canonique.

Notons alors 
$$1 \to \mu_{p^n} \to Z_2 \to Z \to 1$$
la suite exacte obtenue en tirant en arri\`ere l'extension $H_2$ par l'inclusion $Z \to H$. Puisque $Z$ est cyclique d'ordre $p$, l'extension $Z_3$ obtenue \`a partir de $Z_2$ en poussant en avant selon le morphisme naturel $\mu_{p^n} \to \mu_{p^{n+1}}$ est scind\'ee sur $\overline{k}$. On obtient donc finalement un diagramme commutatif dont les lignes sont des extensions centrales de $k$-groupes et dont le carr\'e de droite est cart\'esien :
\begin{displaymath}
\xymatrix{
1 \ar[r] & \mu_{p^{n+1}} \ar[d]^= \ar[r] & Z_3 \ar[d] \ar[r] & Z \ar[d] \ar[r] & 1 \\
1 \ar[r] & \mu_{p^{n+1}} \ar[r] & H_3 \ar[r] & H \ar[r] & 1 \, ,
}
\end{displaymath}
o\`u $H_3$ est obtenu en poussant en avant $H_2$ par le morphisme $\mu_{p^{n+1}} \to \mu_{p^{n}}$ et o\`u l'extension sup\'erieure est scind\'ee comme extension de groupes abstraits.

On remarque alors que pour tous $g, h \in H_3(\overline{k})$ et tout $\gamma \in \Gamma_k$, ${^\gamma [g,h]} = [{^\gamma g}, {^\gamma h}]$. Or le groupe $H$ est constant, donc pour tout $g \in H_3(\overline{k})$ et tout $\gamma \in \Gamma_k$, ${^\gamma g}$ diff\`ere de $g$ par un \'el\'ement du centre de $H_3$. Par cons\'equent, pour tous $g, h \in H_3(\overline{k})$ et tout $\gamma \in \Gamma_k$, ${^\gamma [g,h]} = [{^\gamma g}, {^\gamma h}] = [g,h]$. Donc $\Gamma_k$ agit trivialement sur les commutateurs, donc sur $D(H_3(\overline{k}))$. 

Or le sous-groupe $Z_3(\overline{k})$ de $H_3(\overline{k})$ est contenu dans le sous-groupe de $H_3(\overline{k})$ engendr\'e par $D(H_3(\overline{k}))$ et $\mu_{p^{n+1}}(\overline{k})$, donc on en d\'eduit que l'action de Galois sur $Z_3(\overline{k})$ est triviale (on rappelle que $\mu_{p^{n+1}}(\overline{k}) \subset k^*$). Cela assure que l'extension $Z_3$ est scind\'ee comme extension de $k$-groupes alg\'ebriques. En particulier, cela implique que l'extension 
$$1 \to \G_m \to Z_1 \to Z \to 1 \, ,$$
obtenue en tirant en arri\`ere $H_1$ (et qui co\"incide avec l'extension $Z_3$ pouss\'ee par le morphisme $\mu_{p^{n+1}} \to \G_m$), est scind\'ee, donc a une classe triviale dans $\Ext^c_k(Z,\G_m)$, ce qui conclut la preuve.



\end{proof}

Par cons\'equent, la proposition \ref{prop Brauer} assure que le morphisme canonique $f^* : \Br(X)/\Br(k) \to \Br(Y)/\Br(k)$ induit par $f : Y = G/Z \to X = G/H$ est le morphisme nul.

On rappelle que l'on dispose des applications caract\'eristiques suivantes (pour tout $k$-sch\'ema $T$) qui s'ins\`erent dans un diagramme commutatif :
\begin{displaymath}
\xymatrix{
Y(T) \ar[r] \ar[d]^f & H^1(T,Z) \ar[d] \\
X(T) \ar[r] & H^1(T,H) \, .
}
\end{displaymath}
Via ces applications horizontales (fonctorielles en $T$), les propri\'et\'es d'approximation sur les $k$-vari\'et\'es $X$ et $Y$ se traduisent en des propri\'et\'es d'approximation au niveau de la cohomologie galoisienne des groupes $Z$ et $H$, traductions que l'on utilisera r\'eguli\`erement dans la suite (voir par exemple \cite{H1}).

On a en outre un diagramme commutatif dont les deux premi\`eres lignes sont des suites exactes d'ensembles point\'es :
\begin{displaymath}
\xymatrix{
1 \ar[r] & H^1(k, Z) \ar[r] \ar[d] & H^1(k, H) \ar[r] \ar[d] & H^1(k, H') \ar[r] \ar[d] & H^2(k, Z) \ar[d] \\
1 \ar[r] & P^1(k, Z) \ar[r] \ar[d]^{\partial_Z} & P^1(k, H) \ar[r] \ar[d]^{\partial_H} & P^1(k, H') \ar[r] & P^2(k, Z) \\
& (\Br(Y)/\Br(k))^D \ar[r]^{0} & (\Br(X)/\Br(k))^D &  & \, ,
}
\end{displaymath}
o\`u $P^1(k,H)$ d\'esigne le produit restreint des ensembles $H^1(k_v, H)$ par rapport aux sous-ensembles $H^1(\mathcal{O}_v, H)$.

Raisonnons maintenant par l'absurde pour d\'emontrer le th\'eor\`eme \ref{theo AF} et supposons donc que pour un certain ensemble fini de places $S_0$, on ait $\overline{X(k)}^{S_0} \supset \pi^{S_0} \left( X(\A_k)^{\Br} \right)$.

Soit $(z_v) \in P^1(k, Z)$. On consid\`ere l'image $(h_v) \in P^1(k, H)$ de $(z_v)$. Comme le morphisme $\Ext^c_k(H, \G_m) \to \Ext^c_k(Z, \G_m)$ est nul, on a $\partial_H((h_v)) = 0$. Puisque par hypoth\`ese $\overline{X(k)}^{S_0} \supset \pi^{S_0} \left( X(\A_k)^{\Br} \right)$,
on sait que pour tout ensemble fini $S$ contenant $S_0$, il existe $h^S \in H^1(\mathcal{O}_{k,S}, H)$ tel que $h^S_v = h_v$ pour tout $v \in S \setminus S_0$. Consid\'erons alors l'image $h'^S$ de $h^S$ dans $H^1(\mathcal{O}_{k,S}, H')$. Par construction, on a $h'^S_v = 1$ pour tout $v \in S \setminus S_0$, donc $h'^S \in H^1(\mathcal{O}_{k,S_0}, H')$ (voir par exemple le corollaire A.8 de \cite{GP}). Or l'ensemble de cohomologie \'etale $H^1(\mathcal{O}_{k,S_0}, H')$ est fini, donc en prenant $S$ suffisamment grand, on peut supposer que $h'^S_v = 1$ pour presque toute place $v$, donc $h'^S \in \Sha^1_{\omega}(k,H') = 1$, donc $h'^S = 1$ dans $H^1(\mathcal{O}_{k,S}, H')$ car $H^1(\mathcal{O}_{k,S}, H') \to H^1(k,H')$ a un noyau trivial par restriction-inflation. Cela implique alors que $h^S$ (pour $S$ assez grand) se rel\`eve en $z^S \in H^1(\mathcal{O}_{k,S},Z)$. Puisque pour toute place $v$, l'application $H^1(k_v,Z) \to H^1(k_v, H)$ est injective (car le groupe $H$ est constant et $Z$ est central dans $H$), on en d\'eduit que pour toute place $v \in S \setminus S_0$, on a $z^S_v = z_v$.

En r\'esum\'e, on a montr\'e que pour tout $(z_v) \in P^1(k, Z)$, pour tout $S$ contenant $S_0$, il existe $z^S \in H^1(\mathcal{O}_{k,S},Z)$ tel que pour tout $v \in S \setminus S_0$, on a $z^S_v = z_v$.

Ce dernier fait contredit alors la suite exacte de Poitou-Tate pour $Z$ : en effet, on dispose d'une suite exacte (voir par exemple \cite{M}, th\'eor\`eme I.4.10.(c))
$$H^1(k,Z) \to P^1(k,Z) \to H^1(k, \widehat{Z})^D \, ,$$
qui n'est autre qu'une variante de la premi\`ere colonne du diagramme pr\'ec\'edent.
En particulier, cette suite exacte assure que si l'on choisit $(z_v) \in P^1(k, Z)$ qui n'est pas orthogonal \`a $H^1(k, \widehat{Z})$ et tel que $z_v = 0$ pour toute place $v \in S_0$ (ce qui est toujours possible, puisque $H^1(k, \widehat{Z})$ est infini, alors que $\Sha^1_{\omega}(k, \widehat{Z})$ est fini), alors la propri\'et\'e pr\'ec\'edente ne peut \^etre v\'erifi\'ee.

Cette contradiction assure finalement que l'hypoth\`ese initiale est fausse, donc cela d\'emontre le th\'eor\`eme \ref{theo AF}. 
\end{proof}

\section{Obstruction de Brauer-Manin au principe de Hasse entier}

Au vu du th\'eor\`eme \ref{theo AF}, le th\'eor\`eme 1.4 de \cite{LX} assure qu'avec les notations de la section pr\'ec\'edente, et sous les hypoth\`eses du th\'eor\`eme \ref{theo AF}, il existe un $\mathcal{O}_{k,S_0}$-sch\'ema fid\`element plat, s\'epar\'e de type fini $\mathcal{X}$ pour lequel l'obstruction de Brauer-Manin au principe de Hasse entier n'est pas la seule, autrement dit :
$$\left( \prod_{v \in S_0} X(k_v) \times \prod_{v \notin S_0} \mathcal{X}(\mathcal{O}_v) \right)^{\Br(X)} \neq \emptyset \, ,$$
alors que
$$\mathcal{X}(\mathcal{O}_{k, S_0}) = \emptyset \, .$$

L'objectif de cette section est de construire un exemple explicite d'un tel sch\'ema $\mathcal{X}$, qui est lisse et muni d'une structure d'espace homog\`ene sous un $\mathcal{O}_{k,S_0}$-sch\'ema en groupes \'etendant la structure de $k$-espace homog\`ene sur $X$.

\begin{theo} \label{theo PH}
Soit $p$ un nombre premier, $k$ un corps de nombres contenant les racines $p^{n+1}$-i\`emes de l'unit\'e,
$H$ un groupe fini \emph{non commutatif} d'ordre $p^n$ et $S_0$ un ensemble fini de places de $k$ tel que $p \in \mathcal{O}_{k, S_0}^*$ et $\Pic(\mathcal{O}_{k,S_0}) / p \neq 0$.
Alors il existe une repr\'esentation fid\`ele $\rho : H \to \SL_{d,k}$ et un mod\`ele lisse s\'epar\'e de type fini $\mathcal{X} \to \Spec(\mathcal{O}_{k, S_0})$ de $X := \SL_{d,k}/H$, qui est un $\textup{Spec}(\mathcal{O}_{k, S_0})$-espace homog\`ene de $\SL_{1, \mathcal{A}}$ pour une certaine alg\`ebre d'Azumaya (g\'en\'eriquement d\'eploy\'ee) $\mathcal{A}$ sur $\mathcal{O}_{k, S_0}$, tel que 
$$\left( \prod_{v \in S_0} X(k_v) \times \prod_{v \notin S_0} \mathcal{X}(\mathcal{O}_v) \right)^{\Br(X)} \neq \emptyset \, ,$$
alors que
$$\mathcal{X}(\mathcal{O}_{k, S_0}) = \emptyset \, .$$
En particulier, l'obstruction de Brauer-Manin au principe de Hasse entier sur $\mathcal{X}$ n'est pas la seule.
\end{theo}

\begin{ex} 
Si $p$ est un nombre premier irr\'egulier, on peut consid\'erer le corps $k := \Q(\zeta_{p^{n+1}})$ et l'ensemble $S_0$ form\'e des places archim\'ediennes de $k$ et de l'unique place au-dessus de $p$. Alors on sait que $\Pic(\mathcal{O}_{k,S_0}) / p \neq 0$ (voir \cite{W}, exemple 5) et par cons\'equent toutes les hypoth\`eses du th\'eor\`eme \ref{theo PH} sont v\'erifi\'ees dans cet exemple.
\end{ex}

\begin{proof}
On note $A := \mathcal{O}_{k, S_0}$. On fixe un sous-groupe $Z$ cyclique d'ordre $p$ de $Z(H) \cap D(H)$.
Il existe alors une repr\'esentation fid\`ele de $H$, d\'efinie sur $A$ et dont la dimension $d$ est une puissance de $p$ :
$$\rho : H \to \SL_{d, A} \, ,$$
telle que $\rho(Z)$ soit central dans $\SL_{d,A}$ (on peut par exemple utiliser la repr\'esentation induite par une repr\'esentation fid\`ele de dimension $1$ de $Z$).

Cela induit donc un diagramme commutatif \`a lignes exactes de $A$-sch\'emas en groupes
\begin{displaymath}
\xymatrix{
1 \ar[r] & Z \ar[r] \ar[d]^i & H \ar[r] \ar[d]^{\rho} & H' \ar[r] \ar[d]^{\overline{\rho}} & 1 \\
1 \ar[r] & \mu_{d,A} \ar[r] & \SL_{d,A} \ar[r] & \PGL_{d,A} \ar[r] & 1 \, .
}
\end{displaymath}

Ce diagramme induit le diagramme commutatif suivant :
\begin{displaymath}
\xymatrix{
H^1(A, H') \ar[r]^(.5){\partial_H} \ar[d]^{\overline{\rho}_*} & H^2(A, Z) \ar[d]^{i_*} \\
H^1(A, \PGL_d) \ar[r]^{\partial_{\SL}} & H^2(A, \mu_d) \, .
}
\end{displaymath}

\begin{lem} \label{lem exact}
Sous les hypoth\`eses du th\'eor\`eme \ref{theo PH}, on dispose d'un diagramme commutatif exact :
\begin{displaymath}
\xymatrix{
& 0 \ar[d] & 0 \ar[d] & \\
& \Pic(A) / d \ar[d] \ar[r] & \Pic(A) / m \ar[d] \ar[r] & 0 \\
H^2(A,Z) \ar[r]^{i_*} & H^2(A, \mu_d) \ar[r] & \coker(i_*) \ar[r] & 0 \, ,
}
\end{displaymath}
o\`u $d = p m$, donc $m > 1$ est une puissance de $p$.
\end{lem}

\begin{proof}
Les hypoth\`eses que $p \in A^*$ et que $A$ contienne les racines $p$-i\`emes de l'unit\'e assurent que l'on peut identifier $Z$ avec le $A$-sous-groupe de type multiplicatif $\mu_{p,A}$ de $\mu_{d,A}$. Le lemme r\'esulte alors d'une chasse au diagramme facile dans le diagramme commutatif exact naturel suivant :
\begin{displaymath}
\xymatrix{
\Pic(A) \ar[d]^{[p]} \ar[r]^{=} & \Pic(A) \ar[d]^{[d]} \\
\Pic(A) \ar[d] \ar[r]^{[m]} & \Pic(A) \ar[d] \\
H^2(A, Z) \ar[r]^{i_*} \ar[d] & H^2(A, \mu_d) \ar[d] \\
\Br(A) \ar[r]^= & \Br(A) \, .
}
\end{displaymath}
\end{proof}

En particulier, le lemme \ref{lem exact} implique que le morphisme $i_*$ n'est pas surjectif (sous les hypoth\`eses du th\'eor\`eme, le groupe $\Pic(A)/m$ est non trivial). En revanche, un r\'esultat de Harder (voir \cite{H2}, th\'eor\`eme 4.2.2) assure que l'application $\partial_{\SL}$ est surjective. Donc il existe un $A$-torseur $\mathcal{P}$ sous $\PGL_n$ dont la classe dans $H^1(A, \PGL_d)$ n'est pas dans l'image de $\overline{\rho}_*$, et on peut en outre supposer que l'image de la classe de $\mathcal{P}$ dans $H^2(A, \mu_d)$ provient d'un \'el\'ement de $\Pic(A)/d$ (d'image non nulle dans $\Pic(A)/m$) : voir le diagramme du lemme \ref{lem exact}.

D\'efinissons alors $\mathcal{X}$ comme le quotient de $\mathcal{P}$ par le sous-groupe $H'$ de $\PGL_{d,A}$. Par \cite{Ana} th\'eor\`eme 4.C et \cite{SGA3} proposition 9.2, $\mathcal{X}$ est repr\'esentable par un $A$-sch\'ema lisse s\'epar\'e de type fini. En outre, la proposition 34 du paragraphe I.5.3 de \cite{S} assure que $\mathcal{X}$, qui s'identifie canoniquement au tordu du $A$-espace homog\`ene $\PGL_d / H$ par le $A$-torseur $\mathcal{P}$ sous $\PGL_d$, est un $A$-espace homog\`ene de la $A$-forme int\'erieure ${_\mathcal{P}} \SL_{d, A}$ de $\SL_{d, A}$. De plus, le $A$-sch\'ema en groupes ${_\mathcal{P}} \SL_{d, A}$ est canoniquement isomorphe au $A$-sch\'ema en groupes $\SL_{1, \mathcal{A}}$, o\`u $\mathcal{A}$ est la $A$-alg\`ebre d'Azumaya associ\'ee au $\PGL_d$-torseur $\mathcal{P}$. 

Par construction, l'image de $\mathcal{P}$ dans $H^2(A, \mu_d)$ provient de $\Pic(A)/d$, donc l'image de $\mathcal{P}$ dans $H^1(k, \PGL_d) \cong H^2(k, \mu_d)$ provient de $\Pic(k)/d$, donc est triviale. Cela assure donc que l'alg\`ebre d'Azumaya $\mathcal{A}$ est g\'en\'eriquement triviale. En particulier, la fibre g\'en\'erique $P$ de $\mathcal{P}$ est un $k$-torseur trivial, donc $P(k) \neq \emptyset$ et $P$ est $k$-isomorphe \`a $\PGL_d$. Donc la fibre g\'en\'erique $X = P/H'$ de $\mathcal{X}$ admet un point rationnel et est $k$-isomorphe \`a $\PGL_d / H'$, donc \`a l'espace homog\`ene $\SL_{d,k}/H$.

Montrons que 
$$\left( \prod_{v \in S_0} X(k_v) \times \prod_{v \notin S_0} \mathcal{X}(\mathcal{O}_v) \right)^{\Br(X)} \neq \emptyset \, .$$
Puisque $H^1(\mathcal{O}_v, \PGL_d) = 1$ pour toute place $v \notin S_0$, et puisque pour toute place $v \in S_0$, la classe de $\mathcal{P}_{v}$ dans $H^1(k_v, \PGL_d)$ provient de $\Pic(k_v)/d = 0$, on a :
$$\prod_{v \in S_0} P(k_v) \times \prod_{v \notin S_0} \mathcal{P}(\mathcal{O}_v) \neq \emptyset \, ,$$
donc le morphisme quotient $\pi : \mathcal{P} \to \mathcal{X}$ assure que
$$\prod_{v \in S_0} X(k_v) \times \prod_{v \notin S_0} \mathcal{X}(\mathcal{O}_v) \neq \emptyset \, .$$
En outre, la proposition \ref{prop Brauer} et le lemme \ref{lem morph brauer nul} assurent que le morphisme $\pi^* : \Br(X)/\Br(K) \to \Br(P)/\Br(K)$ est le morphisme nul, donc par fonctorialit\'e, tout \'el\'ement de l'ensemble non vide $\prod_{v \in S_0} P(k_v) \times \prod_{v \notin S_0} \mathcal{P}(\mathcal{O}_v)$ s'envoie dans l'ensemble de Brauer-Manin de $X$.
Cela assure bien que 
$$\left( \prod_{v \in S_0} X(k_v) \times \prod_{v \notin S_0} \mathcal{X}(\mathcal{O}_v) \right)^{\Br(X)} \neq \emptyset \, .$$

Montrons pour finir que $\mathcal{X}(A) = \emptyset$. Supposons qu'il existe $x \in \mathcal{X}(A)$. Alors la fibre de $\pi$ en $x$ d\'efinit un $A$-torseur $\mathcal{Q}$ sous $H'$. Alors la classe de $\mathcal{Q}$ dans $H^1(A,H')$ s'envoie sur la classe de $\mathcal{P}$ dans $H^1(A, \PGL_d)$, i.e. $\mathcal{P} \cong \overline{\rho}_* \mathcal{Q}$. Mais par construction la classe de $\mathcal{P}$ n'est pas dans l'image de $\overline{\rho}_*$, d'o\`u une contradiction. 

Donc finalement $\mathcal{X}(A) = \emptyset$, ce qui conclut la preuve.
\end{proof}

\begin{rem}
On peut v\'erifier facilement que les deux exemples pr\'esent\'es aux th\'eor\`emes \ref{theo AF} et \ref{theo PH}, s'ils ne sont pas expliqu\'es par l'obstruction de Brauer-Manin enti\`ere, le sont en revanche par ce que l'on peut appeler une obstruction de Brauer-Manin \'etale enti\`ere (voir \cite{P}, section 3.3 pour la d\'efinition de la version rationnelle de cette obstruction). En effet, dans le th\'eor\`eme \ref{theo AF}, les points ad\'eliques de $X$ orthogonaux au groupe de Brauer qui ne sont pas dans l'adh\'erence des points rationnels proviennent en fait de points ad\'eliques du rev\^etement \'etale naturel $G/Z \to X$ qui ne sont pas orthogonaux au groupe de Brauer de $G/Z$. De m\^eme, dans le th\'eor\`eme \ref{theo PH}, les points ad\'eliques entiers de $X$ orthogonaux au groupe de Brauer de $X$ proviennent de points ad\'eliques entiers sur le rev\^etement \'etale $P \to X$ de $X$ qui ne sont pas orthogonaux au groupe de Brauer de $P$. Dans les deux cas, il y a donc en quelque sorte une obstruction de Brauer-Manin \'etale enti\`ere qui est strictement plus forte que l'obstruction de Brauer-Manin enti\`ere et qui explique ces contre-exemples : on est donc dans une situation analogue \`a celle de l'exemple 5.10 de \cite{CTW} qui traite de formes quadratiques enti\`eres.
\end{rem}

\section{Un exemple sur le corps des nombres rationnels}

Dans cette partie, on construit un contre-exemple \`a l'approximation forte (resp. au principe de Hasse entier) avec conditions de Brauer-Manin, sur le corps $\Q$ des nombres rationnels (i.e. en se passant de l'hypoth\`ese sur la pr\'esence de racines de l'unit\'e dans le corps de base).

\begin{prop} \label{prop ext} Soit $k$ un corps de nombres.
Si $H = Q_8$ est le groupe des quaternions d'ordre $8$ de centre $Z$, alors le morphisme naturel $\Ext^c_k(H, \G_m) \to \Ext^c_k(Z, \G_m)$ est le morphisme nul. 
\end{prop}

\begin{proof}
On s'inspire de la preuve du lemme \ref{lem morph brauer nul} : on constate qu'il suffit de montrer que le morphisme 
$$\Ext^c_k(H, \mu_8) \to \Ext^c_k(Z, \mu_8)$$
est le morphisme nul. Or la donn\'ee d'une classe dans le groupe $\Ext^c_k(H, \mu_8)$ correspond \`a la donn\'ee d'un groupe abstrait $E$ de cardinal $64$, muni d'une action du groupe de Galois absolu de $\Q$ et s'ins\'erant dans une suite exacte centrale de $k$-groupes
$$1 \to \mu_8 \to E \to H \to 1 \, .$$
En parcourant la liste des groupes d'ordre $64$ avec GAP, on trouve seulement deux groupes qui sont extensions centrales de $H$ par un groupe cyclique d'ordre $8$, en l'occurrence les groupes num\'erot\'es 44 et 126 dans la classification de GAP. Or le second de ces deux groupes est le produit direct $\Z/8\Z \times H$, donc toute action de Galois sur ce groupe compatible avec la suite exacte pr\'ec\'edente est donn\'ee par un cocycle galoisien \`a valeurs dans le module des caract\`eres $\widehat{H}$ de $H$. Par cons\'equent, l'image d'une telle extension centrale par le morphisme $\Ext^c_k(H, \mu_8) \to \Ext^c_k(Z, \mu_8)$ correspond \`a l'image du cocycle par le morphisme $H^1(k,\widehat{H}) \to H^1(k, \widehat{Z})$, qui est clairement le morphisme nul.

Consid\'erons maintenant le cas restant du groupe $E$ num\'ero 44; ce groupe peut \^etre d\'efini par g\'en\'erateurs et relations de la fa\c con suivante : 
$$E = \langle a, b : a^{16} = b^4 = 1, [a,b] = b^2 \rangle \, .$$
Avec les notations pr\'ec\'edentes, le sous-groupe central cyclique d'ordre $8$ est le sous-groupe engendr\'e par $a^2 b^2$. Or on a la relation suivante dans $E$ :
$${^\gamma b^2} = {^\gamma [a,b]} = [^\gamma a,^\gamma b] = [a,b] = b^2 \, ,$$
pour tout $\gamma \in \Gamma_k$, donc l'action de Galois sur $b^2$ est triviale.

Or l'extension centrale de $Z$ par $\mu_8$ obtenue en tirant $E$ en arri\`ere est donn\'ee par les g\'en\'erateurs suivants :
$$1 \to \mu_8 = \langle a^2 b^2 \rangle \to \langle a^2, b^2 \rangle \to \Z / 2 \Z = \langle b^2 \rangle \to 1 \, .$$
Cette suite est donc scind\'ee comme suite de $k$-groupes (puisque l'action de Galois sur $b^2$ est triviale), donc cela assure que l'image de la classe de $E$ par le morphisme $\Ext^c_k(H, \mu_8) \to \Ext^c_k(Z, \mu_8)$ est la classe triviale.

Finalement, on a bien montr\'e que le morphisme $\Ext^c_k(H, \mu_8) \to \Ext^c_k(Z, \mu_8)$ \'etait le morphisme nul, ce qui termine la preuve de la proposition.
 \end{proof}

En reprenant les d\'emonstrations des th\'eor\`emes \ref{theo AF} et \ref{theo PH}, et en y rempla\c cant le lemme \ref{lem morph brauer nul} par la proposition \ref{prop ext}, on obtient les deux corollaires suivants :
\begin{cor}
Soit $k$ un corps de nombres. On consid\`ere l'espace homog\`ene $X := \SL_{4,k}/H$, o\`u $H$ le groupe $Q_8$ des quaternions d'ordre 8 et $H \to \SL_{4,k}$ est une repr\'esentation fid\`ele de dimension $4$ sur $k$.

Alors pour tout ensemble fini $S_0$ de places de $k$, il existe un point ad\'elique $x \in X(\A_k)^\Br$ tel que $\pi^{S_0}(x) \notin \overline{X(k)}^{S_0}$.
Autrement dit, "l'obstruction de Brauer-Manin \`a l'approximation forte" sur (la vari\'et\'e rationnelle) $X$ n'est pas la seule.
\end{cor}

Par exemple, on peut prendre $k=\Q$ dans ce premier corollaire. En particulier, le th\'eor\`eme 1.4 de \cite{LX} assure qu'il existe un $\Z$-sch\'ema fid\`element plat, s\'epar\'e de type fini $\mathcal{X}$ pour lequel l'obstruction de Brauer-Manin au principe de Hasse entier n'est pas la seule, c'est-\`a-dire que 
$$\left(\prod_{p \leq \infty} \mathcal{X}(\Z_p) \right)^{\Br(X)} \neq \emptyset \, \textup{ et } \mathcal{X}(\Z) = \emptyset \, ,$$
et tel que $\mathcal{X}_\Q \cong  \SL_{4,\Q}/Q_8$.

En augmentant l\'eg\'erement le corps de base, on trouve un contre-exemple lisse au principe de Hasse entier qui est lui-m\^eme un espace homog\`ene d'un sch\'ema en groupes r\'eductif :
\begin{cor}
Soit $k$ un corps de nombres et $H$ le groupe $Q_8$ des quaternions d'ordre $8$. Soit $S_0$ un ensemble fini de places de $k$ tel que $2 \in \mathcal{O}_{k, S_0}^*$ et $\Pic(\mathcal{O}_{k,S_0}) / 2 \neq 0$.
Alors il existe une repr\'esentation fid\`ele $\rho : H \to \SL_{4,k}$ et un mod\`ele lisse s\'epar\'e de type fini $\mathcal{X} \to \Spec(\mathcal{O}_{k, S_0})$ de $X := \SL_{4,k}/H$, qui est un $\textup{Spec}(\mathcal{O}_{k, S_0})$-espace homog\`ene de $\SL_{1, \mathcal{A}}$ pour une certaine alg\`ebre d'Azumaya (g\'en\'eriquement d\'eploy\'ee) $\mathcal{A}$ sur $\mathcal{O}_{k, S_0}$, tel que 
$$\left( \prod_{v \in S_0} X(k_v) \times \prod_{v \notin S_0} \mathcal{X}(\mathcal{O}_v) \right)^{\Br(X)} \neq \emptyset \, ,$$
alors que
$$\mathcal{X}(\mathcal{O}_{k, S_0}) = \emptyset \, .$$
En particulier, l'obstruction de Brauer-Manin au principe de Hasse entier sur $\mathcal{X}$ n'est pas la seule.
\end{cor}

Par exemple, on peut v\'erifier que le corps quadratique $k = \Q(\sqrt{-21})$ et  l'ensemble $S_0$ form\'e des places archim\'ediennes et de la place (unique) au-dessus de $2$ dans $k$ ($2$ est ramifi\'e dans $k/\Q$) satisfont les hypoth\`eses de ce second corollaire.


\providecommand{\bysame}{\leavevmode ---\ }
\providecommand{\og}{``}
\providecommand{\fg}{''}
\providecommand{\smfandname}{et}
\providecommand{\smfedsname}{\'eds.}
\providecommand{\smfedname}{\'ed.}
\providecommand{\smfmastersthesisname}{M\'emoire}
\providecommand{\smfphdthesisname}{Th\`ese}

\end{document}